\documentclass[review]{elsarticle}
\usepackage{lineno,hyperref}
\modulolinenumbers[5]

\journal{Nonlinear Analysis: Real World Applications}

\usepackage{graphicx}

\usepackage{alltt}

\usepackage{amsmath}
\usepackage[final]{showkeys}
\usepackage{fancyhdr}
\usepackage{amsmath}
\usepackage{amsfonts,amssymb,amsbsy}
\usepackage{upgreek}
\usepackage{appendix}
\usepackage{color}
\usepackage{caption}
\usepackage{subcaption}
\usepackage[figuresright]{rotating}

\newtheorem{theorem}{Theorem}

\newtheorem{corollary}[theorem]{Corollary}
\newtheorem{remark}[theorem]{Remark}
\newenvironment{proof}[1][Proof]{\textbf{#1.} }{\hfill \raisebox{-0.1em}{$\Box$}\\}

\begin{document}
\begin{frontmatter}
\title{Analysis of Dynamic Pull-in Voltage of a Graphene MEMS Model}
\author[ad1]{Piotr Skrzypacz\corref{mycorr}}
\cortext[mycorr]{Corresponding author}
\ead{piotr.skrzypacz@nu.edu}
\address[ad1]{School of Science and Technology, Nazarbayev University, 
53 Kabanbay Batyr Ave., Astana, 010000, Kazakhstan}
\author[ad2]{Shirali Kadyrov}
\address[ad2]{Suleyman Demirel University, Kaskelen, 040900, Kazakhstan}
\author[ad3,ad1]{Daulet Nurakhmetov}
\address[ad3]{Department of Information System, S. Seifullin Kazakh Agrotechnical University, Zhenis Ave. 62, Astana, 
010011, Kazakhstan}
\author[ad1]{Dongming Wei}
\begin{abstract}
Bifurcation analysis of dynamic pull-in for a lumped mass model is presented. The restoring force of the spring is derived based on the nonlinear constitutive stress-strain law and the driving force of the mass attached to the spring is based on the electrostatic Coulomb force, respectively. The analysis is performed on the resulting nonlinear spring-mass equation with initial conditions. The necessary and sufficient conditions for the existence of periodic solutions are derived analytically and illustrated numerically. The conditions for bifurcation points on the parameters associated with the second-order elastic stiffness constant and the voltage are determined.
\end{abstract}
\end{frontmatter}

{\bf Declarations of interest: none}\\
\section{{I}ntroduction}\label{sec:1}
Pull-in effect occurs in operations of electrostatic controlled micro-electro-mecha\-nical systems (MEMS).
The Pull-in voltage analysis of electrostatically actuated device is
very important for the efficient operation and reliability of the device.
The analysis of the dynamic pull-in voltage of linear materials for MEMS models has been well-established in literature, see, e.g., \cite{Younis}. It is well-known that static pull-in phenomenon occurs when the electrostatic force balances the restoring force at around one-third of the distance between the actuating plate and the base substrate corresponding to linear restoring force.
Numerous research papers have supported these results both experimentally and numerically, see e.g. \cite{ganji1,ganji2}. For a general review of the some comprehensive results on study of pull-in phenomenon and stability analysis in MEMS applications, see, e.g., \cite{Zhang}.
The wonder material graphene has been considered to be an excellent  material candidate for electrostatic MEMS devices.\\[2ex]
 However, it is shown that even for small strains, graphene behaves nonlinearly which results in a nonlinear restoring force in the corresponding lumped-mass models, see e.g. \cite{cadelano, Lee, Lu, malina}. The first mass-spring model for an electrostatically actuated device has been introduced by Nathanson et al.\cite{natanson}. For the mass-spring system, Zhang et al. \cite{Zhang} specify the dynamic pull-in and describe it as the collapse of the moving structure caused by the combination of kinetic and potential energies. In general,
the dynamic pull-in requires a lower voltage to be triggered compared to the static pull-in threshold, see \cite{flores,Zhang}. 

There are relatively few analytical results for lumped-mass models based on nonlinear restoring forces, and it is the purpose of this paper to provide analysis of dynamic pull-in for the lumped-mass model based on the nonlinear elastic behavior of graphene. This is done analytically and exact and explicit formulas for the dynamic pull-in voltages of the nonlinear system are obtained, see Theorem~\ref{thm1} and its corollary.\\ 
\noindent\fbox{\begin{minipage}{\dimexpr\textwidth-4\fboxsep-2\fboxrule\relax}
{\bf Nomenclature}\\[2ex]
\begin{tabular}{ll}
$A,A_c $ & area of the plate and cross-sectional area of the graphene strip\\
$E,{\mathcal{E}}$ & Young's modulus and the energy\\
$d$ & the gap between the moving plate and the substrate plate\\
$D$ & the second-order elastic stiffness constant\\
$F_{res}$ & restoring force of the nonlinear spring\\
$F_C$ & the electrostatic pulling force\\
$K$ & force parameter in the lumped mass model\\
$L$ & length of the graphene strip\\
$m$ & mass of the plate\\
$T$ & time period\\
$t, t_{pull-in} $ & time and pull-in time \\
$V_{DC}$ & voltage\\
$V_{pull-in}$ & pull-in voltage\\
$x$& axial displacement of the plate\\
$x^*$& normalized axial displacement of the plate\\
$x_{max}$& amplitude of the periodic wave\\
$t$ & time\\
$t*$ & normalized time\\
$\varepsilon$ & axial strain\\
$\varepsilon_0$ & electric emissivity\\
$\sigma$ & axial stress\\
$\sigma_{max}$ & the ultimate yield  stress\\
$\alpha$& restoring force parameter\\
\end{tabular}
\end{minipage}}\\[2ex]
In fact, we show that there is a dichotomy: either the initial value problem has a periodic solution or pull-in occurs. The condition which separates periodic solutions from pull-in are demonstrated in terms of the operating voltage, the nonlinear material parameters, the associated geometric dimensions, and the initial conditions. To the best of our knowledge there is no such kind of results for the systems with nonlinear restoring forces. Our results are novel for the system with the nonlinear restoring force resulting from the stress-strain equation for graphene.  We also obtain integral formulas for the time moment when pull-in occurs and the period of the periodic solutions when there is no pull-in.
In this work we derive the conditions for the dynamic pull-in in the case of zero initial conditions in the one degree of freedom (DOF) spring-mass system which represents the graphene-based MEMS. Our mass-spring model can be also considered as one DOF approximations to solutions of MEMS problems which usually require applications of advanced finite element solvers

In Section 2, we present the model, in Sections 3, we demonstrate the analytic pull-in conditions, in Section 4, we illustrate the analytic results numerically, and the Section 5, we draw the conclusions.
\section{{M}odel problem}
For completeness, we introduce the mathematical model describing the motion of an actuating
plate under the elastic and electrostatic Coulomb forces, presented  in \cite{dongming}.
\begin{figure}[htb]
\begin{center}
    \includegraphics[scale=0.5]{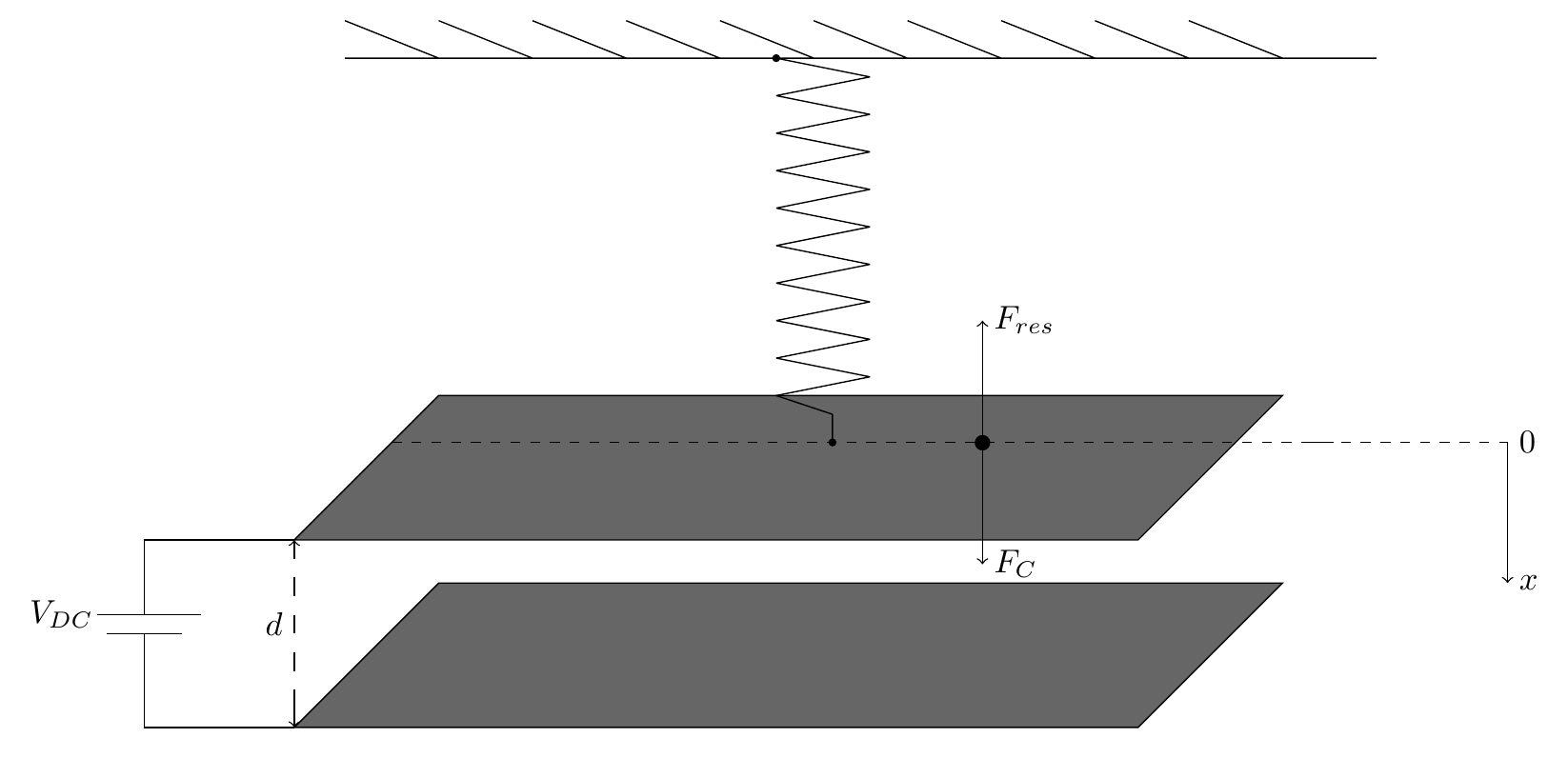}
    \caption{The parallel plate capacitor.}\label{fig_1}
  \end{center}
\end{figure}

Let $A_c$ denote the
cross-sectional area of the graphene strip whose one end is fixed and another one
is attached with a flat plate with mass $m$ and area $A$, see Fig. \ref{fig_1}. Both experimentally and theoretically it is justified in \cite{Lee, Lu} that the graphene material obeys the following constitutive equation
\begin{equation}\label{eq_constitutive_graphene}
\sigma = E\varepsilon - D|\varepsilon|\varepsilon
\end{equation}
where $\varepsilon$,  $\sigma$,  $E$ and $D$ denote the axial strain, axial stress, Young's modulus and second-order elastic stiffness constant, respectively. The constitutive equation \eqref{eq_constitutive_graphene} is valid  for $|\sigma|\le \sigma_{max}$ where the upper bound $\sigma_{max}$ is called the ultimate yield stress of graphene. It determines the second-order elastic stiffness constant as $D=\frac{E^2}{4\sigma_{max}}$. The plate with mass $m$ causes the axial displacement of the graphene based material strip. We will model the graphene strip as a flexible spring whose restoring force is related to the constitution equation \eqref{eq_constitutive_graphene} as follows
\begin{equation}
F_{res} = -EA_c\frac{x}{L}+DA_c\left|\frac{x}{L}\right|\frac{x}{L}
\end{equation}
where $x$ denotes the axial displacement from the equilibrium level of the plate, the strain is approximated by $x/L$ and $L$ is the length of the graphene strip. The moving plate is also subject to the electrostatic Coulomb force which occurs if we place another parallel substrate plate and impose the voltage $V_{DC}$, see Fig.~\ref{fig_1}. The electrostatic Coulomb force is given by
\begin{equation}
F_{C} = \frac{\varepsilon_0 AV_{DC}^2}{2(d-x)^2}
\end{equation}
where $d$ is the gap between two plates, the relative distance between the moving and fixed plates is $d-x$ and $\varepsilon_0$ is the electric emissivity. By Newton’s second law of motion, the vertical displacement variable $x$ in the lumped-mass nonlinear spring model satisfies the following nonlinear equation
\[
m \ddot{x}=F_{res}+F_{C}\,,
\]
so
\begin{equation}\label{eq_dim}
m\ddot{x}+EA_c\frac{x}{L}-DA_c\left|\frac{x}{L}\right|\frac{x}{L}=\frac{\varepsilon_0 AV_{DC}^2}{2(d-x)^2}
\end{equation}
where $\ddot{x}=\frac{d^2x}{dt^2}$. If the electrostatic force dominates significantly the restoring force, then the moving plate can approach or touch the fixed bottom plate. In this case, the so called pull-in occurs.

Let us introduce dimensionless quantities
\begin{equation}\label{eqn:less}
x^*=\frac{x}{d}\,,\quad t^*=t\sqrt{\frac{EA_c}{mL}}\,,\quad
K=\frac{\varepsilon_0 ALV_{DC}^2}{2EA_cd^3}\,,\quad \alpha=\frac{Dd}{EL}=\frac{Ed}{4L\sigma_{max}}\,.
\end{equation}
For simplicity of notation we omit the asterisks. Then, we obtain the following dimensionless form of \eqref{eq_dim}
\begin{equation}\label{eq_nondim}
\ddot{x}+x-\alpha |x|x = \frac{K}{(x-1)^2}\,,
\end{equation}
where $K\ge 0$ and $\alpha\ge 0$. We prescribe the following initial conditions $x(0)=x_0$, $\dot{x}(0)=x_0'$ where $x_0<1$. The initial value problem \eqref{eq_nondim} can be rewritten as the first order system
\begin{equation}
\label{first_order_system}
\left\{
\begin{array}{rcl}
\dot{x} & = & y\,,\\
\dot{y} & = & -x+\alpha |x|x+\frac{K}{(1-x)^2}
\end{array}
\right.
\end{equation}
with initial values $x(0)=y(0)=0$.
The case of $\alpha=K=0$ corresponds to the classical harmonic oscillator. The case of $K=0$ and $\alpha>0$ has been discussed in \cite{Cveticanin}, and the solution can be determined in terms of elliptic Jacobi functions.
 Recently, the case of time dependent $K$ has been discussed by \cite{Torres}. The static pull-in voltage analysis for the case of $\alpha=0$ has been stated in \cite{Younis}. We notice that the value
\[
K=\frac{4}{27}
\]
from \cite{Younis} corresponds to the static pull-in voltage $V_{pull-in}=\sqrt{\frac{8kd^3}{27\varepsilon_0 A}}$, and it exceeds the value which will be determined in the following Section.\\
\begin{figure}[htb]
\begin{center}
\begin{subfigure}[htb]{0.40\textwidth}
 \begin{center}
    \includegraphics[scale=0.33]{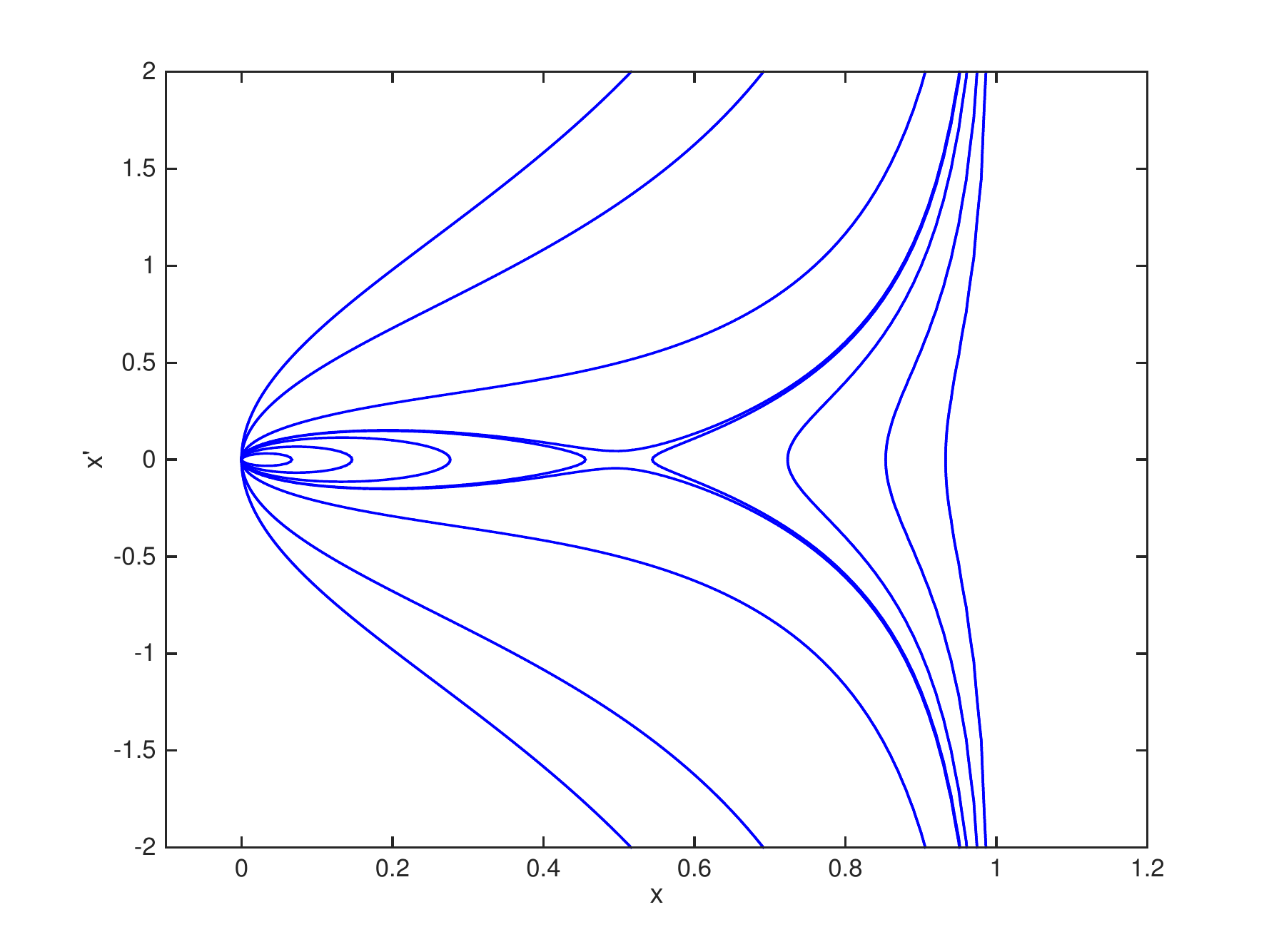}
    \caption{$\alpha=0$}
  \end{center}
   \end{subfigure}\qquad
\begin{subfigure}[htb]{0.40\textwidth}
 \begin{center}
    \includegraphics[scale=0.33]{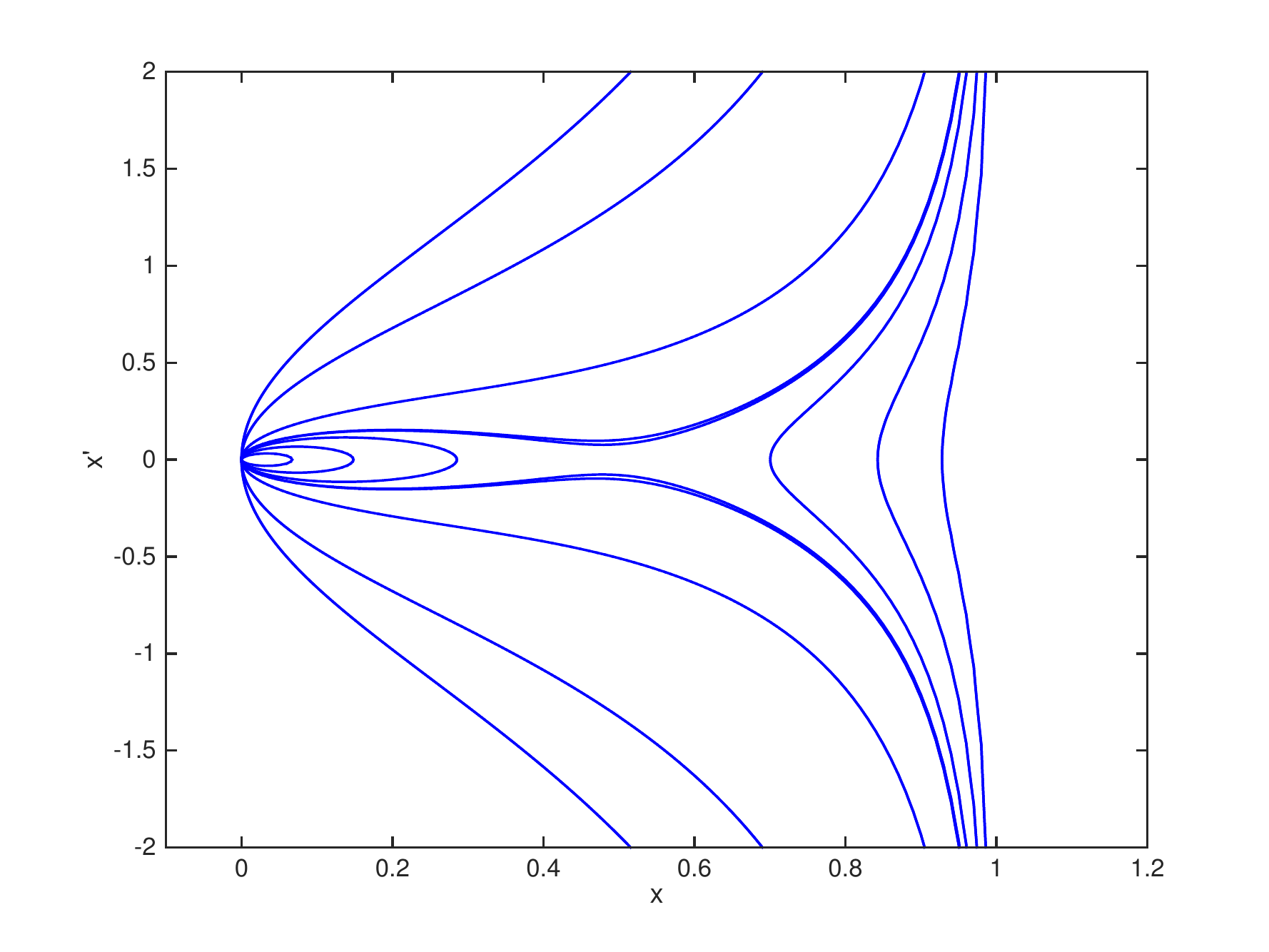}
    \caption{$\alpha=0.1$}
  \end{center}
\end{subfigure}\\[3ex]
\begin{subfigure}[htb]{0.40\textwidth}
  \begin{center}
    \includegraphics[scale=0.33]{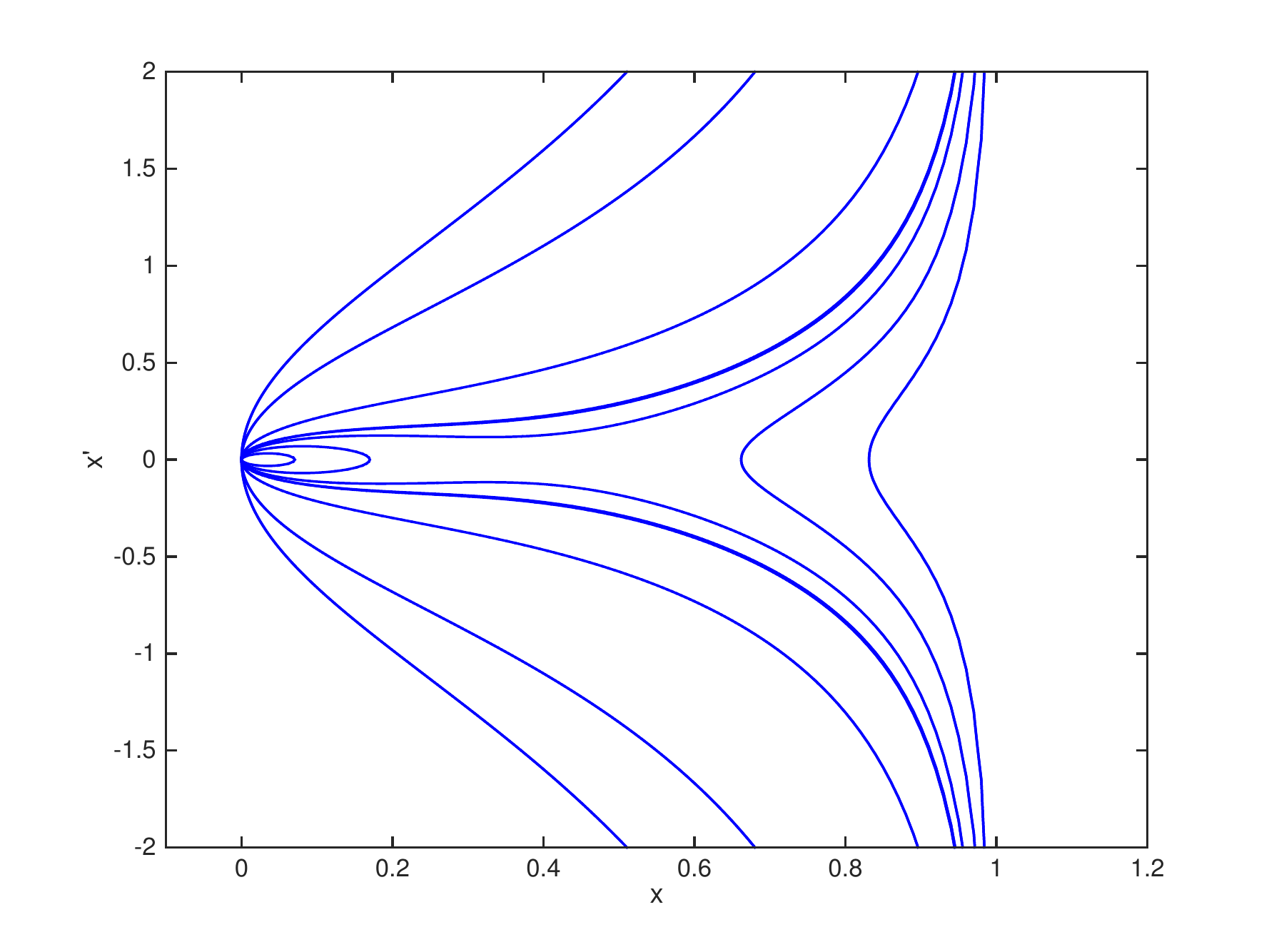}
    \caption{$\alpha=1$}
  \end{center}
\end{subfigure}\qquad
\begin{subfigure}[htb]{0.40\textwidth}
  \begin{center}
    \includegraphics[scale=0.33]{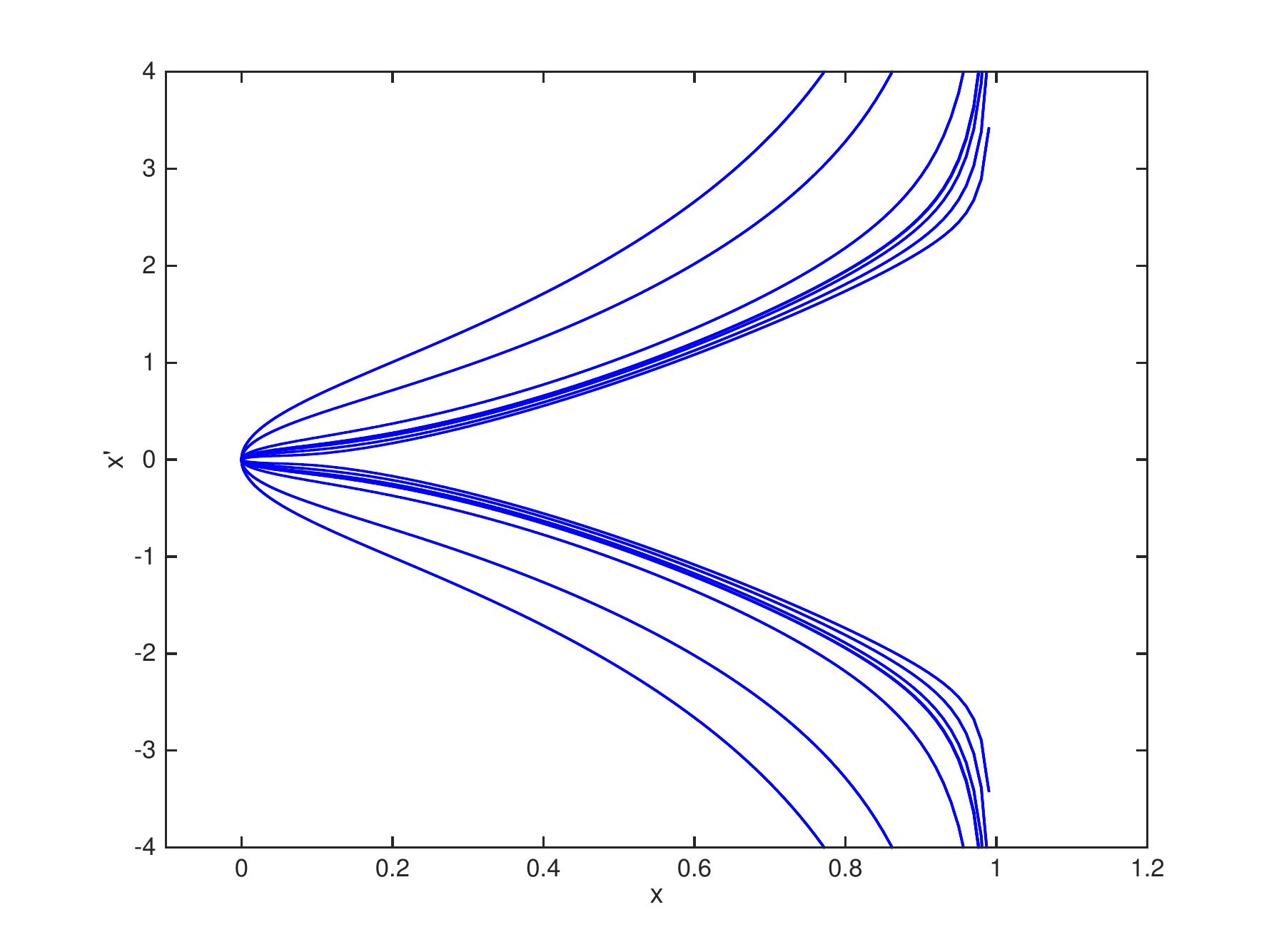}
    \caption{$\alpha=10$}
  \end{center}
\end{subfigure}
	\caption{Phase trajectories for $K=2$, $1$, $0.25$, $0.126$, $0.124$, $0.1$, $0.0625$, $0.03125$ and different values of $\alpha$.}\label{fig_phase_diagrams}
  \end{center}
\end{figure}
\section{Analysis of dynamic pull-in voltage}
Multiplying \eqref{eq_nondim} by $\dot{x}$ and integrating with respect to time, we get the conservation of energy
\[
{\mathcal{E}}(t)=\frac{1}{2}(\dot{x}(t))^2+\frac{1}{2}x^2(t)-\frac{1}{3}\alpha |x(t)|x^2(t)-\frac{K}{1-x(t)}\,,
\]
i.e.,
\begin{equation*}
\label{eq_energy}
\frac{1}{2}(\dot{x})^2+\frac{1}{2}x^2-\frac{1}{3}\alpha |x|x^2-\frac{K}{1-x} = C
\end{equation*}
where $C=\frac{1}{2}(x'_0)^2+\frac{1}{2}x_0^2-\frac{1}{3}\alpha |x_0|x_0^2-\frac{K}{1-x_0}$. Thus,
\begin{equation}\label{eq_dxdt^2}
(\dot{x})^2=-x^2 + \frac{2}{3}\alpha |x|x^2+\frac{2K}{1-x}+(x_0')^2+x_0^2
-\frac{2}{3}\alpha |x_0| x_0^2-\frac{2K}{1-x_0}\,.
\end{equation}
Our analysis is based on the phase diagrams for the case $x_0=x_0'=0$. The phase diagrams for several values of $\alpha$ and $K$ are presented in Fig.~\ref{fig_phase_diagrams}. The closed orbits in phase diagrams represent periodic solutions. We see in Fig.~\ref{fig_phase_diagrams} that the periodic solutions are expected for small parameter values of $\alpha\ge 0$ and $K>0$. The following theorem determines the range of the positive parameters $\alpha$ and $K$ for which
the pull-in occurs.

\begin{theorem}\label{thm1} Let $\alpha>0$, $\mu=\sqrt{4\alpha^2-6\alpha+9}$ and
\begin{equation*}
\kappa(\alpha)=\frac{(2\alpha+3-\mu)(-4\alpha^2+24\alpha-9+2\alpha\mu+3\mu)}{648\alpha^2}\,.
\end{equation*}
The initial value problem \eqref{eq_nondim} with zero initial values has a periodic solution if and only if
$K \le \kappa(\alpha)$
whereas the pull-in occurs if $K>\kappa(\alpha)$.
\end{theorem}

\begin{proof}
	As in \eqref{first_order_system}, we can think of \eqref{eq_dxdt^2} as a system of first order differential equations by letting $y=\dot x$. Then, the solution $x(t)$ is periodic if and only if the phase diagram, $x$ vs. $y$, produces a closed curve. With zero initial conditions, $x_0=x_0'=0$, this means that it is necessary and sufficient for the energy equation \eqref{eq_dxdt^2} to have a closed curve. This is the case when
\[
f_{\alpha,K}(s):=-s^2 + \frac{2}{3}\alpha s^3+\frac{2K}{1-s}-2K
=\frac{s(-\frac{2}{3}\alpha s^3+(\frac{2}{3}\alpha +1)s^2-s+2K)}{1-s}
\]
has a root in $(0,1)$, see Fig.~\ref{fig_phase_diagrams} for illustration. Note that $f_{\alpha,K}>0$ when $s$ approaches to $0$ and $1$ in $(0,1)$. Hence, due to Mean Value theorem, the existence of a root in $(0,1)$ is equivalent to existence of a local minimum of $f_{\alpha,K}$ in $(0,1)$ that is at most 0.   
Therefore, $h_{\alpha,K}(s)= -\frac{2}{3}\alpha s^3+(\frac{2}{3}\alpha +1)s^2-s+2K$ must have a non-positive minimum at some $s\in (0,1)$. To find the critical points, we compute $h'_{\alpha,K}(s)=-2\alpha s^2+2(\frac{2}{3}\alpha+1)s-1$. Then, considering the second derivative, we see that $h_{\alpha,K}$ attains its minimum at the smallest critical point $s_1=(2\alpha/3+1-\sqrt{(2\alpha/3+1)^2-2\alpha})/(2\alpha)$ which is in $(0,1)$ and we must have $h_{\alpha,K}(s_0) \le 0.$ 
\begin{figure}[htb]
\begin{center}
    \includegraphics[scale=0.7]{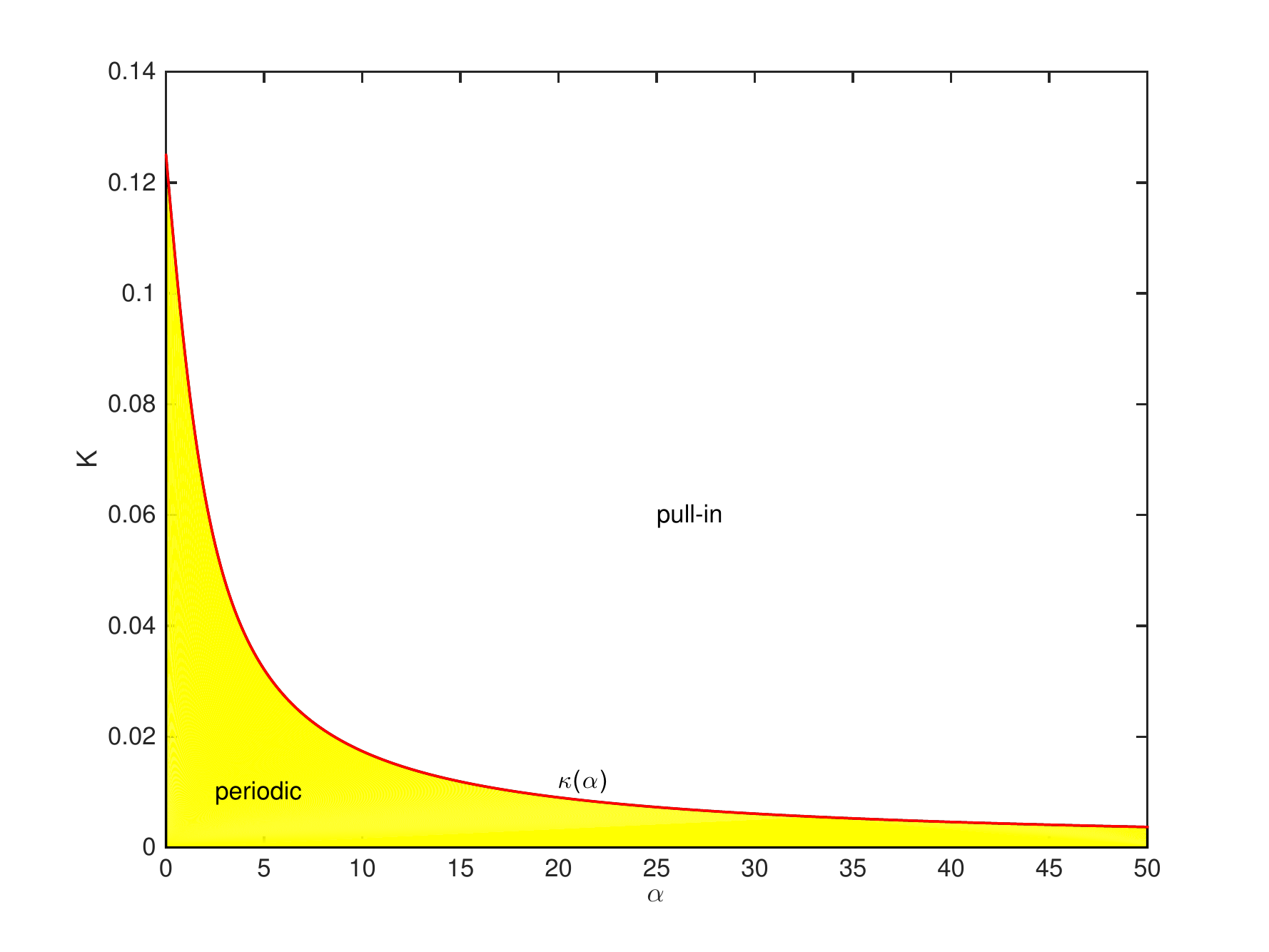}
\caption{Values of parameters $\alpha>0$ and $K$ for pull-in and periodic solutions in the case of $x_0=x_0'=0$.}\label{fig_periodic}
  \end{center}
\end{figure}
This results in
\begin{equation*}
\begin{split}
K&\le -\frac{1}{648\alpha^2}(-2\alpha-3+\sqrt{4\alpha^2-6\alpha+9})\times\\
&\qquad \times (-4\alpha^2+24\alpha-9+2\alpha\sqrt{4\alpha^2-6\alpha+9}+3\sqrt{4\alpha^2-6\alpha+9})\,.
\end{split}
\end{equation*}
The condition for pull-in solutions is equivalent to $h_{\alpha,K}(s_1)>0$.
\end{proof}

As an immediate consequence of Theorem~\ref{thm1} we obtain the exact formula for dynamic pull-in DC voltage.
\begin{corollary}
The pull-in voltage for the lumped-mass nonlinear spring model \eqref{eq_dim} satisfies
$$V_{pull-in}=\sqrt{\frac{2E A_c d^3\kappa(\alpha)}{\varepsilon_0 A L}},$$
where $\alpha=\frac{Dd}{EL}$ and $\kappa(\alpha)$ as in Theorem~\ref{thm1}.
\end{corollary}

\begin{proof}
This obviously follows from Theorem~\ref{thm1} and identities in \eqref{eqn:less}.
\end{proof}
\begin{remark}
Let $K^0_{pull-in}=\lim\limits_{\alpha\to 0^+} \kappa(\alpha$). We have
$K^0_{pull-in}= 1/8$, so we recovered the case of linear spring, i.e. $\alpha=0$. The solution to problem \eqref{eq_nondim} with initial values $x_0=x_0'=0$ is periodic if $K<K^0_{pull-in}$, and the pull-in occurs if $K>K^0_{pull-in}$.
\end{remark}
The values $(\alpha,K)$ for which the initial value problem \eqref{eq_dxdt^2} has periodic or pull-in solutions are presented in Fig.~\ref{fig_periodic}. In the linear case $\alpha=0$ we have pull-in if $K>K^0_{pull-in}$ with $K^0_{pull-in}=\frac{1}{8}$ which is smaller than the value $K=\frac{4}{27}$ for the static pull-in obtained by \cite{Younis}. Notice that for $\frac{4}{27}>K>\frac{1}{8}$ the pull-in still occurs. The general case of arbitrary initial values $x_0\ge 0$ and $x_0'$ will be discussed in forthcoming works.
\begin{figure}[htb]
\begin{center}
\begin{subfigure}[htb]{0.40\textwidth}
 \begin{center}
    \includegraphics[scale=0.33]{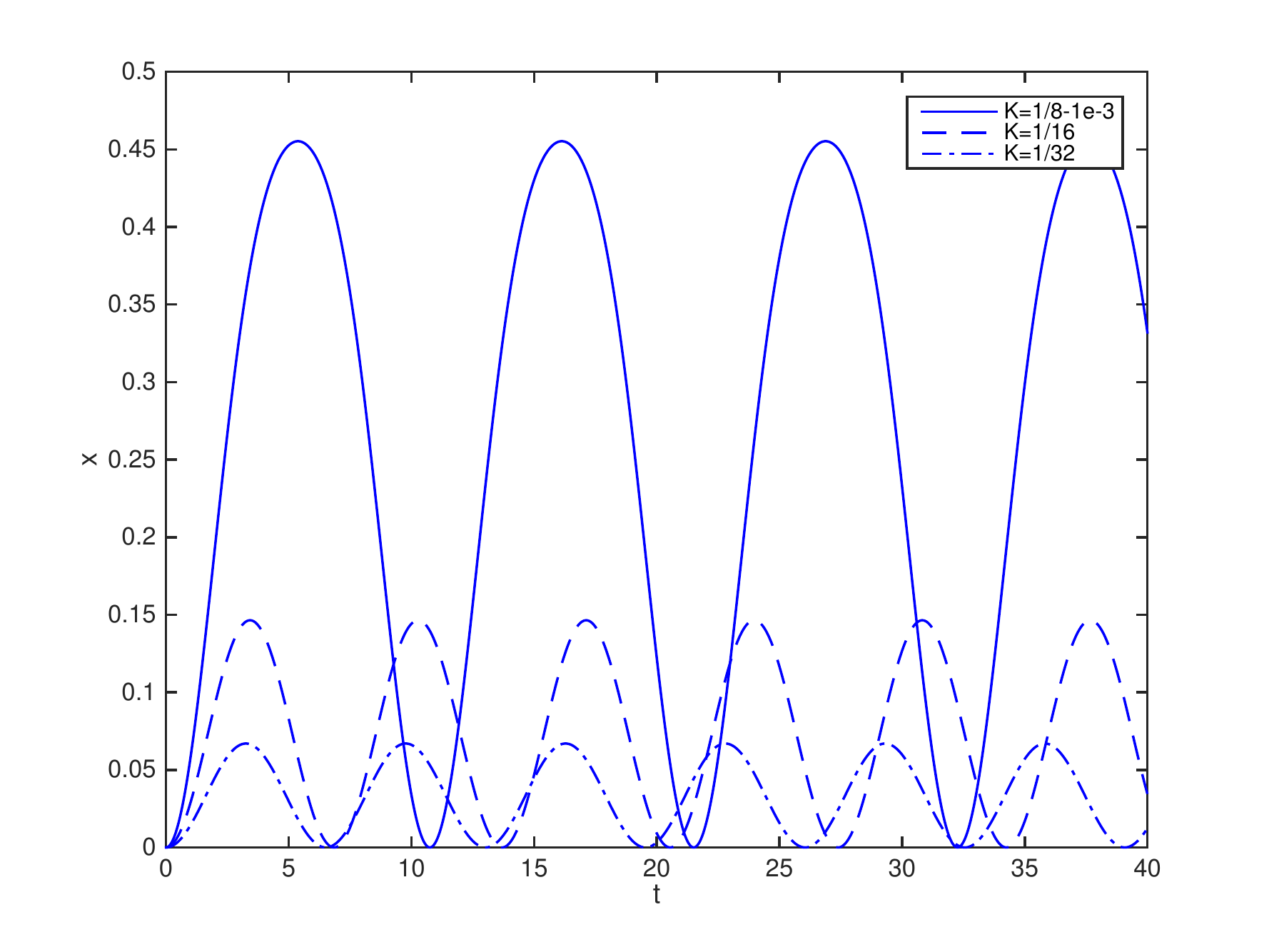}
    \caption{$\alpha=0$}
  \end{center}
   \end{subfigure}\qquad
\begin{subfigure}[htb]{0.40\textwidth}
 \begin{center}
    \includegraphics[scale=0.33]{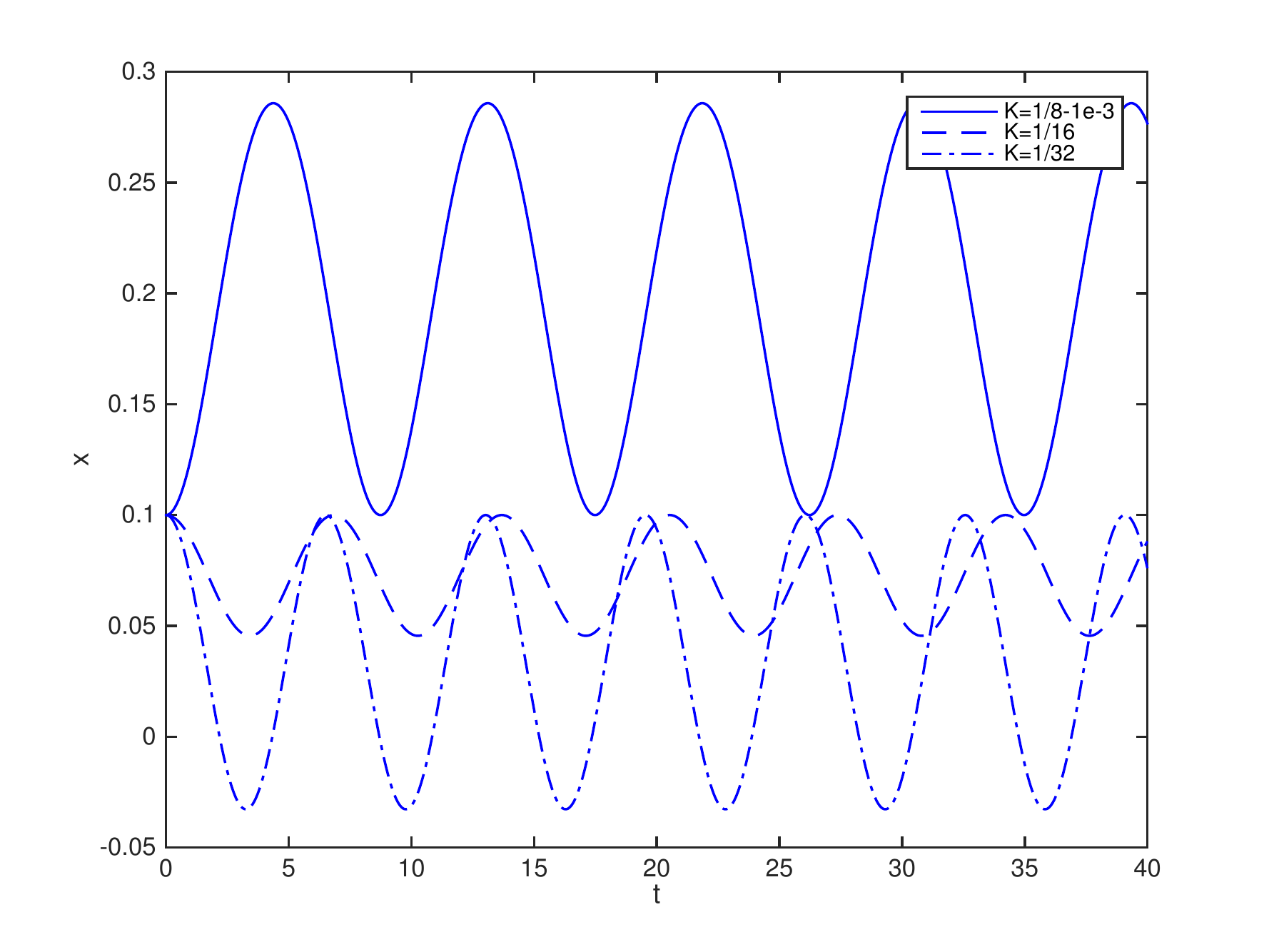}
    \caption{$\alpha=0$, $x_0=0.1$, $x_0'=0$ }
  \end{center}
\end{subfigure}\\[2ex]
\begin{subfigure}[htb]{0.40\textwidth}
  \begin{center}
    \includegraphics[scale=0.33]{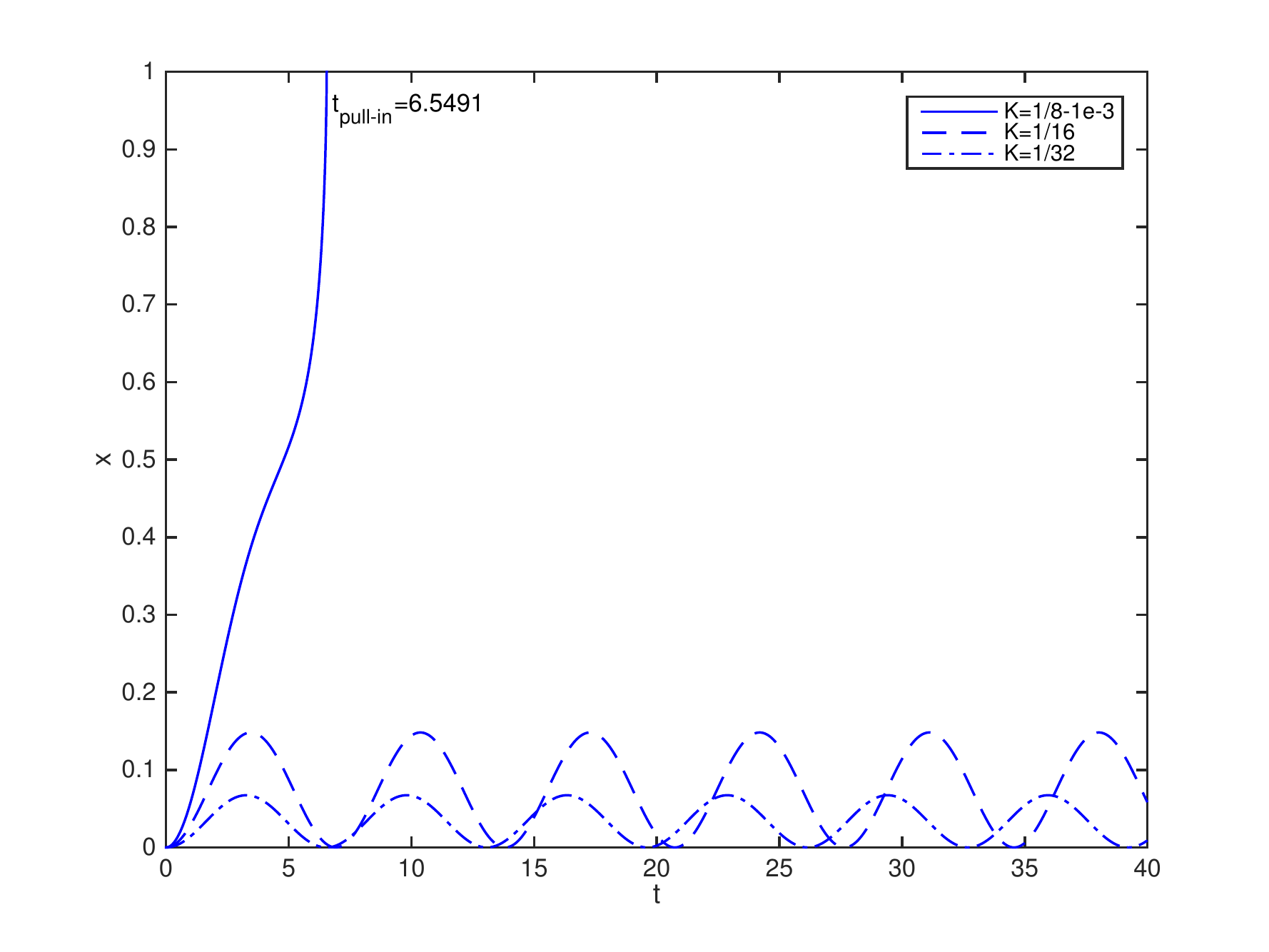}
    \caption{$\alpha=0.1$}
  \end{center}
\end{subfigure}\qquad
\begin{subfigure}[htb]{0.40\textwidth}
  \begin{center}
    \includegraphics[scale=0.33]{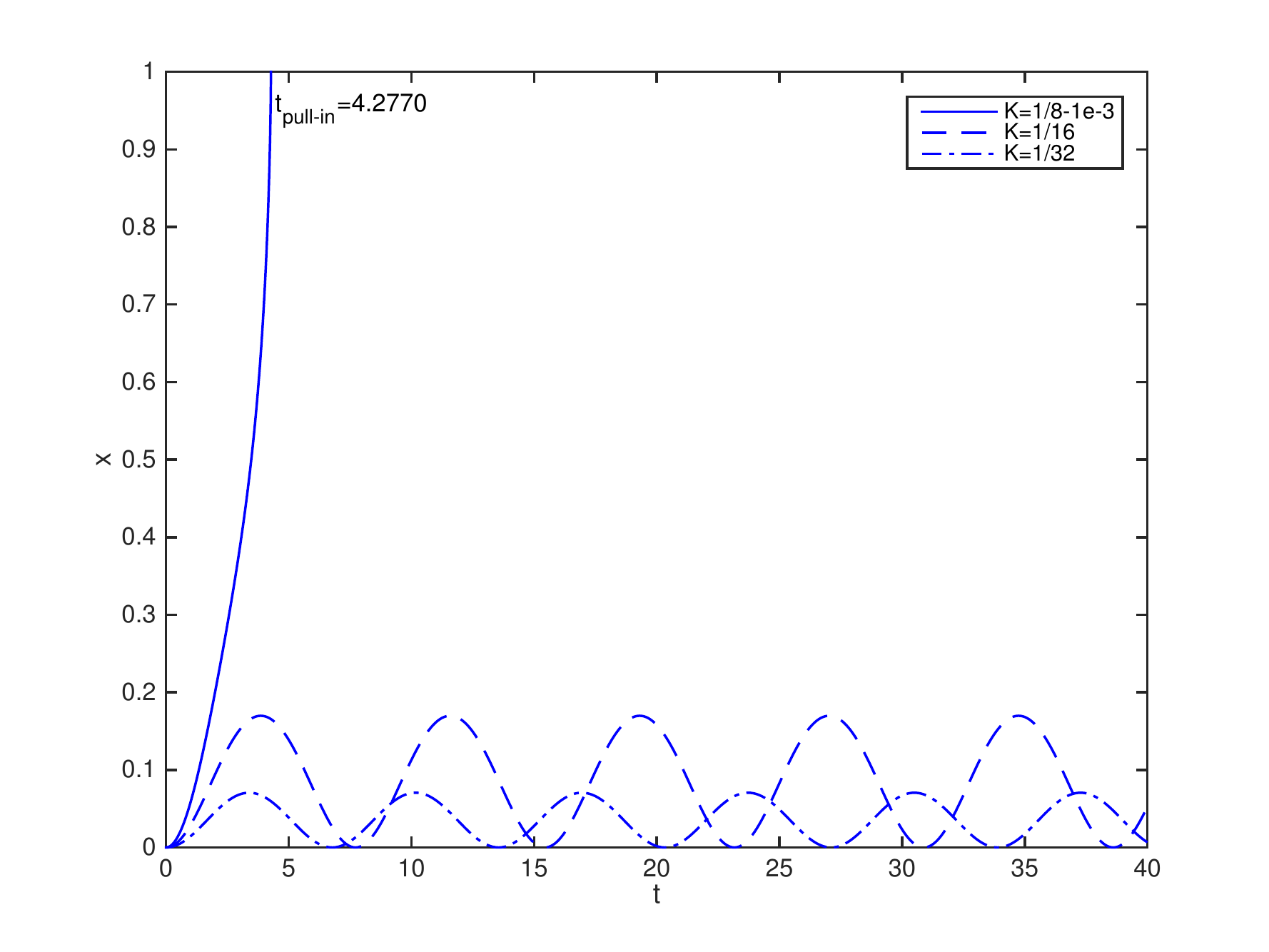}
    \caption{$\alpha=1$}
  \end{center}
\end{subfigure}
\caption{Solution profiles for different values of $K$ and $\alpha$.}\label{fig_plots}
  \end{center}
\end{figure}
The exact pull-in time $t_{pull-in}$ and the period $T$ can be computed as follows
\begin{equation*}
t_{pull-in}=\int\limits_{0}^1\frac{ds}{\sqrt{-s^2+\frac{2}{3}\alpha s^3+\frac{2K}{1-s}-2K}}
\end{equation*}
and
\begin{equation*}
T=\int\limits_{0}^{x_{max}}\frac{2ds}{\sqrt{-s^2+\frac{2}{3}\alpha s^3+\frac{2K}{1-s}-2K}}\,,
\end{equation*}
where $x_{max}\in (0,1)$ is the root of $h_{\alpha,K}(s)$. We notice that $x_{max}$ also represents the amplitude. In the case of linear spring, we have
\begin{equation*}
t_{pull-in}=\int\limits_{0}^1\frac{ds}{\sqrt{-s^2+\frac{2K}{1-s}-2K}}\,,\quad K>\frac{1}{8}\,,
\end{equation*}
and
\begin{equation*}
T=\int\limits_{0}^{\frac{1}{2}-\frac{1}{2}\sqrt{1-8K}}\frac{2ds}{\sqrt{-s^2+\frac{2K}{1-s}-2K}},\quad 0<K<\frac{1}{8}\,.
\end{equation*}
The above values are verified by numerical experiments in the next Section.
\section{{N}umerical experiments}
In order to justify the predicted behavior of the solutions $x(t)$ to the model problem \eqref{eq_nondim},
we employed the standard Matlab ODE solver \texttt{'ode23s'} using the following options for high accuracy numerical 
solutions 
\begin{center}
\texttt{'options = odeset('RelTol',1e-10,'AbsTol',[1e-12 1e-12])'}.
\end{center}
This ODE solver is based on the embedded Runge-Kutta method, see \cite{Shampine}.
\begin{figure}[htb]
\begin{center}
\begin{subfigure}[htb]{0.42\textwidth}
 \begin{center}
    \includegraphics[scale=0.36]{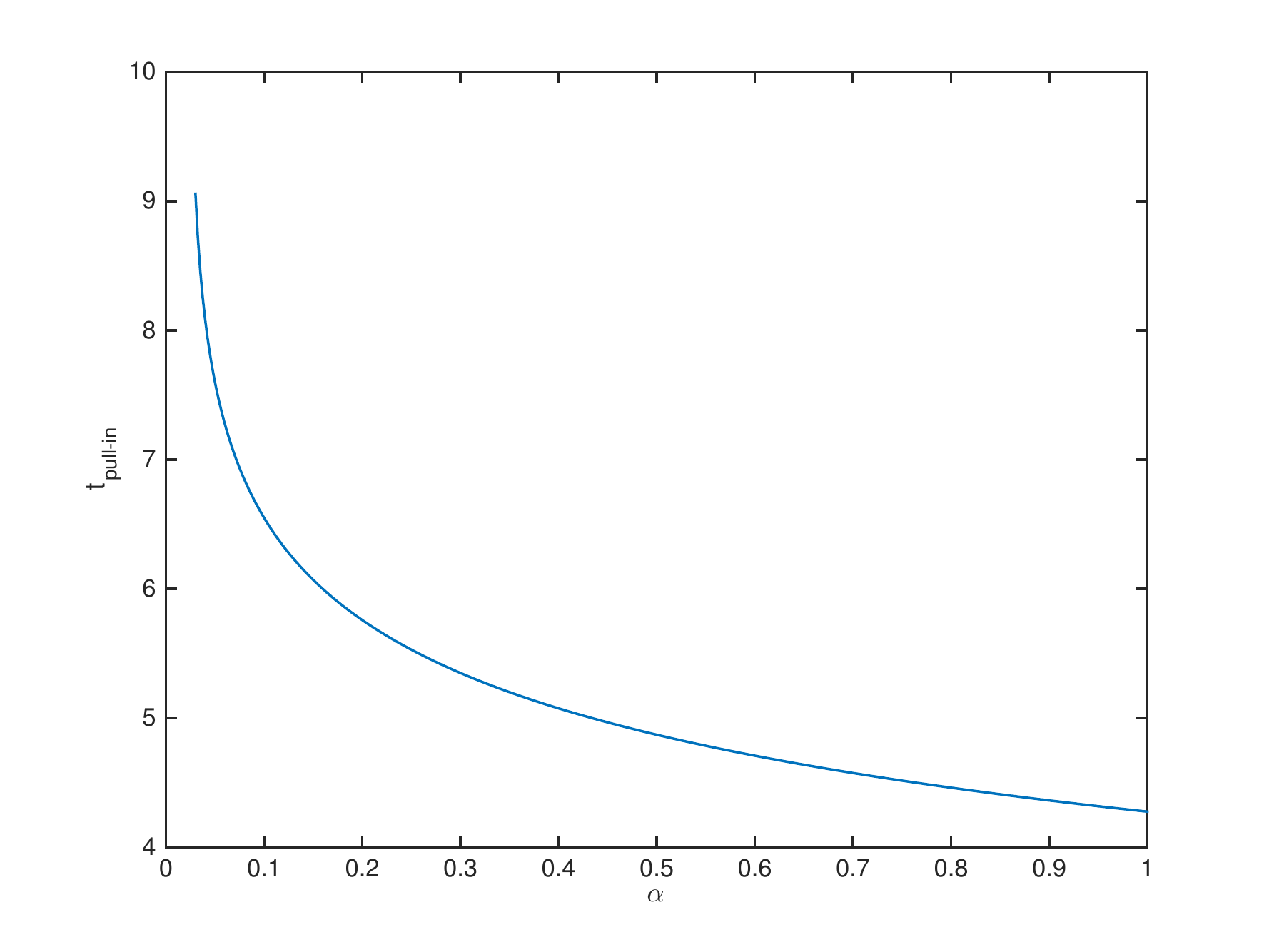}
    \caption{$K=0.124$}
  \end{center}
   \end{subfigure}\qquad
\begin{subfigure}[htb]{0.32\textwidth}
 \begin{center}
    \includegraphics[scale=0.36]{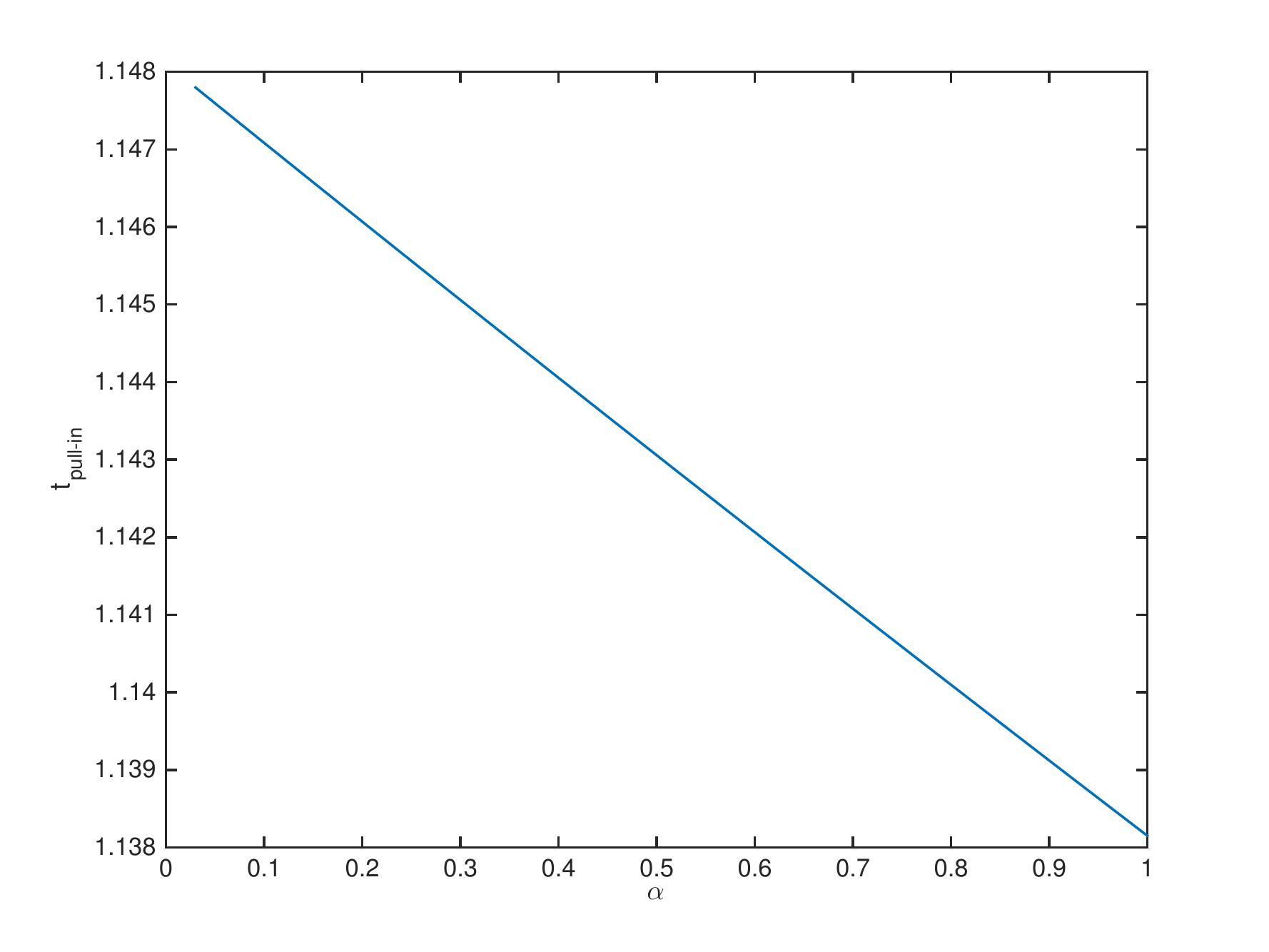}
    \caption{$K=1$ }
  \end{center}
\end{subfigure}
\caption{Pull-in time for $\alpha\in [0.03,1]$ and $K=0.124,\; 1$}. \label{fig_pull_in}
  \end{center}
\end{figure}
Several solution profiles $x(t)$ for different sets of parameter values are presented in Fig.~\ref{fig_plots}, where the periods and pull-in times depend on the parameters $\alpha$, $K$ and the initial values. 
In Fig.~\ref{fig_pull_in} we numerically demonstrated how the pull-in time decreases for $K=0.124$ and $K=1$ when the parameter $\alpha\in [0.03,1]$ increases. 
\section{{C}onclusions} Pull-in conditions for lumped-mass models subject to the electrostatic force with the nonlinear restoring force arising from the constitutive stress-strain equation for graphene are obtained. Specific conditions for pull-in phenomenon to occur in the model are presented analytically in terms of the operating voltage, the nonlinear material parameters, the geometric dimensions, and the initial conditions. Numerical illustrations of the analytic solutions are also presented. The results obtained in this work are novel and can be useful for design of some MEMS made of graphene.\\

\paragraph{{\bf Acknowledgment}}{~This research was supported by the Nazarbayev University ORAU grant ``Modeling and Simulation of Nonlinear Material Structures for Mechanical Pressure Sensing and Actuation Applications``.}

\section*{References}
\bibliography{pull_in}

\end{document}